\documentclass[a4paper, 12pt]{article}
\usepackage{amsmath, amssymb, mathptmx, enumerate}
\usepackage{amsthm}
\usepackage{thmtools, tikz}

\declaretheoremstyle[notefont=\bfseries,notebraces={}{},%
    headpunct={},postheadspace=0.5em]{mystyle}
   
\newcommand\blfootnote[1]{%
  \begingroup
  \renewcommand\thefootnote{}\footnote{#1}%
  \addtocounter{footnote}{-1}%
  \endgroup
}
    
\declaretheorem[style=mystyle,numbered=no,name=Theorem]{theorem}
\declaretheorem[style=mystyle,numbered=no,name=Lemma]{lemma}

\renewenvironment{proof}{{\bf Proof:} }{\hfill\par\medskip}
\newcommand{\prt}[2]{\frac{\partial #1}{\partial #2} }
\newcommand{\de }{differential equation}
\begin{document}
\begin{center}
\section*{EXACT DIFFERENTIAL EQUATIONS AND HARMONIC FUNCTIONS}
{\bf Azubuike Henry C.*}, \ {\bf Dagiloke Joseph O.} \\
*Department of Mathematics,\\
*Federal University of Technology, Owerri, Nigeria\\
{*henry.azubuike@futo.edu.ng}
\end{center}
\begin{abstract}
In this work, we investigate some connections between exact \de s and harmonic functions and in particular, we obtain necessary and sufficient conditions for which exact equations admit harmonic solutions. As an application, we consider the orthogonal trajectories of harmonic functions, and among other results we obtain that the Cauchy-Riemann equations and the non-vanishing of the first partial derivatives are sufficient for any two curves to be orthogonal trajectories of each other. All curves throughout the work are restricted to the $xy-$plane.
\end{abstract}
\section*{1\quad Introduction}
Harmonic functions -- the solutions of Laplace's equation play a very crucial role in many areas of mathematics, physics and engineering [1]. \par\smallskip 
Differential equations on the other hand are indispensable tools in applied mathematics. In this work, we look at some connections between harmonic functions (and their conjugates) and exact \de s -- an idea that stemmed from the similarities in the process of finding harmonic conjugates and that of solving exact \de s. We ask the questions: Can harmonic functions be solutions of exact \de s? Can we classify exact \de s that admit harmonic solutions? Given a harmonic function, can we find the exact equation satisfied by its conjugate without knowing the conjugate itself? 
\par\smallskip 
The above questions are answered in this work, and most importantly a necessary and sufficient condition for an exact \de\ to admit a harmonic solution is proved. \par\smallskip  As an application, we consider the orthogonal trajectories of harmonic functions and among other results, we obtained that the orthogonal trajectories of harmonic functions are necessarily harmonic and furthermore, we also showed that harmonic functions and their conjugates are orthogonal trajectories of each other. \par\smallskip 
While the analyses carried out in this work is restricted to $\mathbb{R}^2$, the $xy-$plane, we may have been able to show that certain problems concerning harmonic functions can be posed in terms of exact \de s and importantly, the harmonic conjugates found by the existing methods may as well be a particular solution to the more general problem of finding the orthogonal trajectories of their harmonic functions, as the latter problem offers suggestions that certain arbitrary functions of harmonic functions are also harmonic.
\section*{2\quad Preliminaries}
Some of the existing results that are used in this work are stated in this section. The basic theory and proofs may be found in the references.\par\medskip 
We begin with an exact \de.
\begin{theorem}[2.1]{\bf [2]}\ \
{\it Let the functions $M,\ N,\ M_x$ and $N_y$, where the subscripts denote the partial derivatives, be continuous in a simply-connected region ${R}$, then the \de 
\[M(x,y)\mathrm{d}x + N(x,y)\mathrm{d}y = 0 \tag{2.1}\]
is an exact \de\ in ${R}$ if and only if \[M_y(x,y) = N_x(x,y)\tag{2.2}\]
at each point of ${R}$. That is, there exists a function $f$ satisfy \[ f_x(x,y) = M(x,y), \qquad f_y(x,y) = N(x,y)\]
if and only if $M$ and $N$ satisfy equation (2.2).}
\end{theorem} 
Next, we turn to harmonic functions. A twice continuously differentiable function $u$ defined on a subset of $\mathbb{R}^n$ is harmonic if it solves the Laplace's equation: \[\Delta u = 0\] 
where $\Delta = \mathcal{D}_1^2 + \mathcal{D}_2^2 +\cdots +\mathcal{D}_n^2 $ and $\Delta_j^2$ denotes the second partial derivative with respect to the $j$th coordinate variable [1]

In this paper, we restricted our functions (domains) to the $xy-$plane. Consequently, the Laplace equation in two variables is stated as \[\prt{^2u}{x^2} + \prt{^2u}{y^2} = 0\]
and therefore, any real-valued function of two real variables $x$ and $y$ that has continuous first and second order partial derivatives in a domain $\mathfrak{D}$ and satisfies Laplace's equation is said to be harmonic in $\mathfrak{D}$.[5]\par\smallskip  
Inasmuch as we try to avoid the word analyticity, the Cauchy-Riemann equations could not be avoided in the work, and so we state it here as a necessary condition for analyticity.
\begin{theorem}[2.2]{\bf [5]}{ \it \quad 
Suppose $f(z) = u(x,y) + iv(x,y)$ is differentiable at a point $z = x + iy.$ Then, at the point $z$, the first order partial derivatives of $u$ and $v$ exists and satisfy the Cauchy-Riemann equations:}
\[\prt{u}{x} = \prt{v}{y}\qquad \mbox{ and }\qquad \prt{u}{y} = -\prt{v}{x} \tag{2.3}\]\qed
\end{theorem}
\section*{3\quad Main work} 
 Here, we look more critically at the relationship between exact differential equations and harmonic functions with the corresponding harmonic  \mbox{conjugate}.
\begin{theorem}[3.1]{\it
Let $M\mathrm{d}x + N\mathrm{d}y = 0$ be an exact \de\ in (2.1) with harmonic solution $f(x,y) = c$, where $M, N, f$ are at least once  differentiable, then  the harmonic conjugate of f is given as: \[g(x, y) = \int M\partial y - \int\left[ N + \frac{\partial}{\partial x}\int M\partial y \right] \mathrm{d}x + c\]}
\end{theorem}
\begin{proof}
Firstly, we observe that since $f(x,y)$ is the solution of the exact \de, then \[\frac{\partial f}{\partial x} = M \Rightarrow f= \int M\partial x + g(y)\] differentiating with respect to $y$, keeping $x$ constant  \begin{align*}
\frac{\partial f}{\partial y}&  = \frac{\partial }{\partial y}\int M\partial x + g'(y) = N\\
& \Rightarrow g'(y) = N - \frac{\partial }{\partial y}\int M\partial x\\
\therefore f(x,y) = \int M\partial x &+ \int\left[N-\frac{\partial }{\partial y}\int M\partial x\right] \mathrm{d}x = c \tag{3.1}
\end{align*} 
\end{proof}
Next, we proceed to find the harmonic conjugate $g(x,y)$ of $f(x, y)$. Since the conjugate harmonic function $g$ must satisfy the Cauchy-Riemann equations: \[u_x = v_y \quad \mbox{ and } \quad u_y = -v_x\]
in this case, \quad 
\newcommand{\dst}{\displaystyle}
$\dst\frac{\partial f}{\partial x} = \frac{\partial g}{\partial y}$\quad and\quad $\dst\frac{\partial f}{\partial y} = -\frac{\partial g}{\partial x}$.\par\smallskip 
We must have: 
\begin{align*}
\frac{\partial g}{\partial y} & = \frac{\partial}{\partial x}\left[\int M\partial x + \int\left[N - \frac{\partial }{\partial y}\int M\partial x\right] \mathrm{d}y\right]\\
& = \frac{\partial }{\partial x}\int M\partial x + \frac{\partial }{\partial x}\int N\mathrm{d}y - \frac{\partial}{\partial x}\int \frac{\partial }{\partial y}\int M\partial x \mathrm{d}y\\
& = M + \frac{\partial }{\partial x} \int N\mathrm{d}y - \frac{\partial }{\partial y}\int\left(\frac{\partial}{\partial x}\int M\partial x\right) \mathrm{d}y\\
& = M + \frac{\partial}{\partial x}\int N\mathrm{d}y - \frac{\partial}{\partial y}\int M\partial y\\
& = M + \frac{\partial }{\partial x}\int N\mathrm{d}y - M = \frac{\partial }{\partial x}\int N\partial y 
\end{align*}
Partial integration with respect to $y$ gives: \[ g(x,y)  = \int\left( \frac{\partial }{\partial x}\int N\partial y\right) \partial y + h(x) \tag{3.2}
\]
Differentiating the above with respect to $x$ partially gives: \[ 
\prt{g}{x}= \prt{}{x}\int \prt{}{x}\int N\partial y \partial y + h'(x)\]
But by the Cauchy-Riemann equations, $\dst\prt{g}{x} = -\prt{f}{y}$, and so we have: \par\smallskip 
\begin{align*}
\prt{}{x}\int \prt{}{x} \int N\partial y \partial y + h' (x)\\ 
& = -\prt{}{y}\int M\partial x - \prt{}{y}\int N\partial y + \prt{}{y}\int \prt{}{y}\int M\partial x \partial y\\
& = - \prt{}{y}\int M\partial x - N + \prt{}{y}\int \left( \prt{}{y}\int M\partial y\right) \partial x \\
& = - \prt{}{y}\int M\partial x - N + \prt{}{y}\int M\partial x = -N
\end{align*}
that is: \[ h'(x) = -N -\prt{}{x}\int \prt{}{x}\int N \partial y \partial y\]
integrating both sides with respect to $x$:
\[h(x) = -\int\left[ N + \prt{}{x}\int\prt{}{x} \int N\partial y \partial y \right] \tag{3.3}\]
Substituting the above into {(3.2)}:
\[g(x,y) = \int \prt{}{x}\int N\partial y\partial y - \int\left[ N + \prt{}{x}\int\prt{}{x}\int N \partial y \partial y\right] \mathrm{d}x + c. \tag{3.4}\]
But we know that $f$ is the solution of the \de\ $M\mathrm{d}x + N\mathrm{d}y = 0$. This means that $\dst\prt{f}{y} = N \Rightarrow f(x,y) = \int N\partial y + c_1(x)$,\par\smallskip  
taking the partial derivative with respect to $x$: 
\begin{align*}
\prt{f}{x} &= M(x, y) = \prt{}{x}\int N\partial y + c_1'(x)\\
& \Rightarrow \prt{}{x}\int N\partial y = M - c_1'(x)
\end{align*}
substituting this into (3.4), we have: 
\begin{align*}
g(x, y) & = \int (M - c_1'(x))\partial y - \int \left[ N +  \prt{}{x}\int (M- c_1'(x))\partial y\right] \mathrm{d}x + c\\
& = \int M\partial y - c_1'(x)y - \int\left[N + \prt{}{x}\int M\partial y - \prt{}{x}\int c_1'(x)\partial y\right] \mathrm{d}x + c\\
& = \int M\partial y - c_1'(x)y - \int \left[N + \prt{}{x}\int M\partial y - c_1''(x)y \right]\mathrm{d}x + c\\
& = \int M \partial y - c_1'(x)y - \int \left[N + \prt{}{x}\int M \partial y \right]\mathrm{d}x + \int c_1''(x)y\ \mathrm{d}x + c \\
& = \int M\partial y - c_1'(x)y - \int \left[ N + \prt{}{x}\int M \partial y \right] \mathrm{d}x + c_1'(x)y + c \\
& = \int M\partial y - \int \left[ N + \prt{}{x}\int M\partial y\right] \mathrm{d}x + c 
\end{align*}
This concludes the proof of Theorem 3.1. \hfill\qed \par\smallskip 
We show here that the quantity in the integral sign : $\dst N + \prt{}{x}\int M\partial y$\ is independent of $y$. That is: \par\smallskip 
$\dst\prt{}{y}\left[ N + \prt{}{x} \int M\partial y \right]$ should be zero.
\begin{align*}
\prt{}{y}\left[ N + \prt{}{x} \int M\partial y \right] & = \prt{N}{y} + \prt{}{y}\prt{}{x} \int M \partial y\\
& = \prt{N}{y} + \prt{}{x}\prt{}{y}\int M\partial y\\
& = \prt{N}{y} + \prt{M}{x} = 0
\end{align*}
since $f(x, y)$ is harmonic. The harmonicity of the solution of the \de\ guarantees that \[ \prt{}{y}\left[ N + \prt{}{x} \int M \partial y \right] =0. \tag{3.5}\]
From the above, we have been able to obtain the harmonic conjugate of a harmonic function in terms of integrals related to an exact \de. \par\smallskip 
The existence and characterization of such exact \de s admitting harmonic solutions form the crux of this work. Just before that, we approach the problem: given a harmonic function $u(x,y)$, can we obtain a corresponding exact \de\ for which $u(x,y)$ is a solution? \par 
The answer to this question is a resounding Yes! In essence, we claim that every harmonic function generates an exact ordinary \de.
\begin{lemma}[3.1]
{\it Let $u(x, y)$ be a harmonic function, then $u(x,y)$ solves the exact\par \mbox{equation}: \[M\mathrm{d}x + N\mathrm{d}y = 0 \tag{3.6}\] 
where $M(x, y) = u_x(x, y)$ and $N(x, y) = u_y(x,y)$.
}
\end{lemma}
\begin{proof}
The proof of the above lemma relies on the fact that (12) is the total \de\ of $u(x,y)$. Since $u(x,y)$ is harmonic, we are guaranteed of its differentiablility. \hfill\ $\qed$
\end{proof}
In Theorem 3.1, we have shown that if the exact \de\ \[M\mathrm{d}x + N\mathrm{d}y = 0\] has a harmonic solution $u(x,y)$, then the conjugate $v(x,y)$ of $u(x,y)$ is given as: \[v(x,y) = \int M \partial y - \int\left[ N + \prt{}{x}\int M \partial y\right] \mathrm{d}x + c.\]
Now, comparing with (3.6) above this corresponds to: \[v(x,y) = \int u_x\partial y - \int \left[ u_y + \prt{}{x}\int u_x \partial y\right] \mathrm{d}x + c. \tag{3.7}\]
We ask the question, which exact \de\ does the harmonic conjugate $v(x,y)$ satisfy? From Lemma 3.1, we know that $v(x,y)$ solves: 
\[v_x(x,y)\mathrm{d}x + v_y(x,y)\mathrm{d}y = 0\]
we attempt to obtain the above \de\ in terms of the given function $u(x,y)$, so from (3.7), we have:
\begin{align*}
\prt{}{x} \int u_x\partial y -  \prt{}{x}\int\left[u_y + \prt{}{x}\int u_x\partial y\right] \mathrm{d}x & = v_x(x,y) \tag{3.8}\\
\prt{}{y} \int u_x\partial y -  \prt{}{y}\int\left[u_y + \prt{}{x}\int u_x\partial y\right] \mathrm{d}x = v_y(x,y) \tag{3.9}
\end{align*} 
From (3.8), 
\begin{align*}
v_x(x,y) & = \prt{}{x}\int u_x\partial y - \prt{}{x}\int\left[u_y + \prt{}{x}\int u_x\partial y \right] \mathrm{d}x\\
& = \prt{}{x}\int u_x\partial y - \prt{}{x}\int u_y\mathrm{d}x - \prt{}{x}\int\prt{}{x}\int u_x\partial y \mathrm{d}x\\
& = \prt{}{x}\int u_x\partial y - u_y - \prt{}{x}\int \left( \prt{}{x}\int u_x\partial x\right) \partial y\\
& = \prt{}{x}\int u_x\partial y - u_y - \prt{}{x}\int u_x \partial y = -u_y
\end{align*}
Also, from (3.9):
\begin{align*}
v_y(x,y) &= \prt{}{y}\int u_x\partial y - \prt{}{y}\int\left[ u_y + \prt{}{x}\int u_x\partial y\right] \mathrm{d}x
\end{align*}
but we have shown in (3.5) that the quantity under the integral sign in the last term is independent of $y$ and consequently, the entire last term above vanishes. This is because $u(x,y)$ is harmonic.\par\smallskip  Therefore;
\[v_y(x,y) = \prt{}{y}\int u_x \partial y = \prt{}{y}\int M \partial y = M\]
In all, we have: \[v_x = -u_y = -N;\quad v_y = u_x = M\]
Thus we have the \de: \[ -u_y\mathrm{d}x + u_x\mathrm{d}y = 0\]
corresponding to \[-N\mathrm{d}x + M\mathrm{d}y = 0\tag{3.10}\]
The equation is exact since \[ \prt{(-N)}{y} = \prt{(-u_y)}{y} -u_{yy} \]
and \[ \prt{(M)}{x} = \prt{u_x}{x} = u_{xx}\]
\[\mbox{  implies }\quad \prt{(-N)}{y} = \prt{M}{x} \Leftrightarrow -u_{yy} = u_{xx}\Leftrightarrow u_{xx}  + u_{yy} = 0  \hspace*{4.5cm} \]
We have now formulated a relationship between harmonic functions and their exact ordinary \de s. So far given a harmonic function $u(x,y)$, we can construct the exact \de\ for which $u(x,y)$ is a solution and every harmonic function has this property (in fact, every function which is at least once differentiable with continuous derivatives has this property).\par\smallskip  
Now, since every harmonic function $u(x,y)$ has  a harmonic conjugate, $v(x,y)$, we have been able to construct $v(x,y)$ in terms of integrals involving $u(x,y)$ (or equivalently the terms in the \de\ associated with $u(x,y)$). We have also obtained an exact conjugate \de\ associated with $u(x,y)$ which is simply the exact \de\ (3.10) for which $v(x,y)$ is a solution.\par\smallskip 
All these have been achieved from the standpoint of $u(x,y)$, it remains to obtain a class of exact \de s that have harmonic solutions. 
\begin{theorem}[2.]
{\it Let $M(x,y)\mathrm{d}x + N(x,y)\mathrm{d}y = 0$\ be a given exact \de. Then, the equation admits a harmonic solution if and only if \[M_x + N_y = 0.\]}
\end{theorem}
\begin{proof}
The \de\ $M\mathrm{d}x + N\mathrm{d}y = 0$ is exact means that \[M_y = N_x.\]
We first show that if $M_x + N_y = 0$, then the equation has a harmonic solution. The solution of the \de\ is given in equation (1) as \[f(x,y) = \int M\partial x + \int\left[ N- \prt{}{y}\int M\partial x\right] \mathrm{d}y = c\]
we shall show that $f$ is harmonic.\par 
$f_x = M$ and $f_y = N$ and so $f_{xx}=M_x,\quad f_{yy} = N_y$.\par
$f_{xx} + f_{yy} = M_x + N_y = 0,$ hence $f(x,y)$ is harmonic.\par\smallskip Conversely, we assume that the \de\ has a harmonic solution $f(x,y)$ given as in (1) above. Since $f$ is harmonic, it  follows that \[f_{xx} + f_{yy} = 0\quad \Leftrightarrow \quad M_x + N_y = 0\]
This concludes the proof of the necessary and sufficient condition for the exact \de\ to have a harmonic solution.\hfill \qed
\end{proof}\vfill\newpage
We therefore have the diagram: 
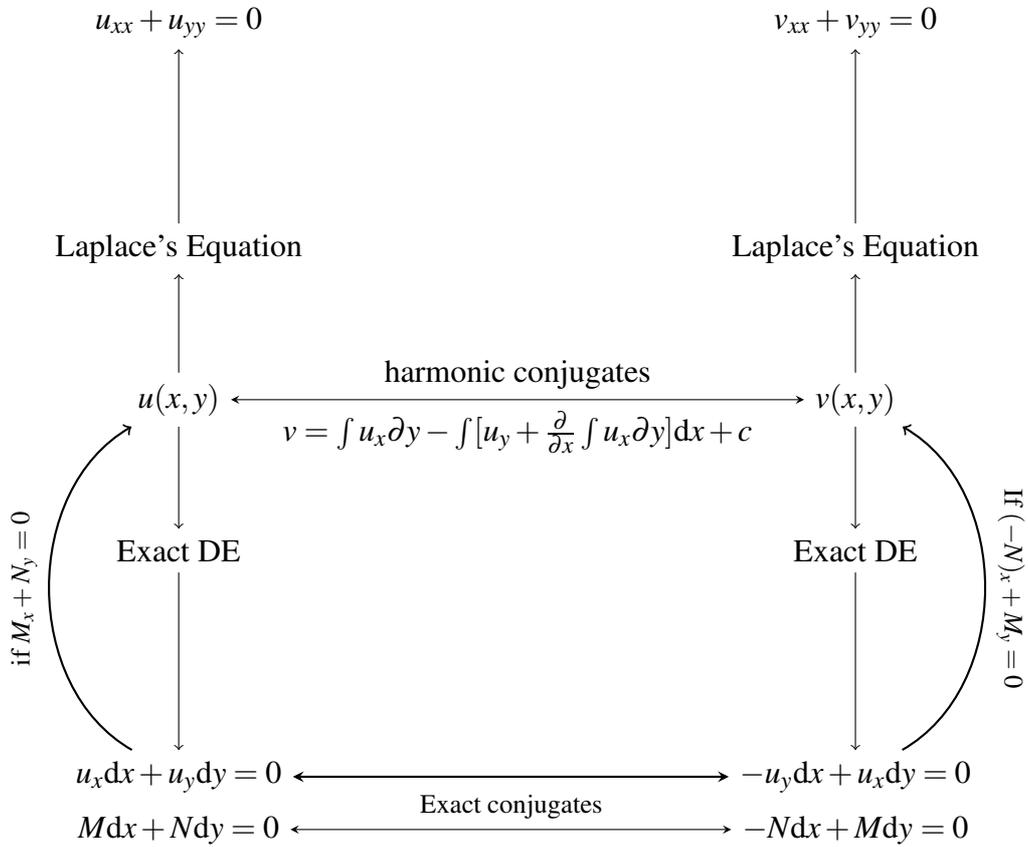
\begin{center}
\begin{tikzpicture}
\usetikzlibrary{arrows}
\tikzset{nd/.style={node distance = 8.9cm}}
\node (uxy) {$u(x,y)$};
\node [right of = uxy, nd] (vxy) {$v(x,y)$};
\node [above of = uxy, node distance = 2cm] (le){Laplace's Equation} node [above of = le, node distance=3cm] (ab){$u_{xx}+u_{yy} = 0$};
\node[below of = uxy, node distance=2cm] (ede){Exact DE} node [below of = ede, node distance=3cm] (udxy){$u_x\mathrm{d}x + u_y\mathrm{d}y = 0$} node [below of = udxy, node distance=0.7cm](nmx){$M\mathrm{d}x + N\mathrm{d}y = 0$}; 
\draw [->] (uxy) -- (le);
\draw[->] (le) -- (ab);
\draw[->] (uxy) -- (ede);
\draw[->](ede) -- (udxy);
\draw[->, thick] (udxy) [bend left = 60] to  node[sloped,rotate=180, above] {\footnotesize if $M_x + N_y = 0$} (uxy);
\node[above of = vxy, node distance=2cm](le2){Laplace's Equation} node[above of = le2, node distance=3cm](bd){$v_{xx} + v_{yy} = 0$};
\node[below of = vxy, node distance=2cm](ede2){Exact DE};
\node[below of = ede2, node distance=3cm] (udyx) {$-u_y\mathrm{d}x + u_x\mathrm{d}y = 0$};
\node[below of = ede2, node distance=3.7cm] (nmy) {$-N\mathrm{d}x + M\mathrm{d}y = 0$};
\draw[->] (vxy) -- (le2);
\draw[->](le2) -- (bd);
\draw[->](vxy) -- (ede2);
\draw[->] (ede2) -- (udyx);
\draw[->, thick] (udyx) [bend right = 60] to node[sloped, above,  rotate=180 ]{\footnotesize If  $(-N)_x + M_y = 0$} (vxy);
\draw[stealth-stealth] (uxy) --node[above]{harmonic conjugates} node [below] {$v=\int u_x\partial y - \int [u_y + \frac{\partial}{\partial x}\int u_x\partial y]\mathrm{d}x + c$} (vxy);
\draw[stealth-stealth, thick] (udxy) -- (udyx);
\draw[stealth-stealth] (nmy) --node[above]{\footnotesize Exact conjugates} (nmx); 
\end{tikzpicture}
\end{center}
\blfootnote{DE = Differential Equations}
\vfill\newpage 
\section*{4 \quad Orthogonal Trajectories}
Suppose that we have a family of curves given by $f(x,y,c_1) = 0$, where one curve corresponds to each $c_1$ in some range of values of the parameter $c_1$. In certain applications, it is found desirable to know what curves have the property of intersecting a curve of the family $f(x,y,c_1)$ at right angles whenever they do intersect.\par\smallskip 
By definition, when all the curves of a family of curves $f(x,y,c_1) =0$ intersects orthogonally all the curves of another family $g(x,y,c_2) = 0$ then the families are said to be orthogonal trajectories of each other.[3]\par\smallskip 
The method of finding orthogonal trajectories of given curves relies on the fact that if two curves are to be orthogonal, then at each point of intersection, the slopes of the curves must be negative reciprocals of each other. In this section, we attempt to investigate the orthogonal trajectories of harmonic functions.\par\smallskip 
Let $u(x,y,c) = 0$ be any twice differentiable function with at most one arbitrary constant $c$. We can obtain the exact \de\ associated with \\ $u(x,y,c) = 0$ by taking the total derivative: 
\[M\mathrm{d}x + N\mathrm{d}y = 0\tag{4.1}\]
with $M = u_x$ \ and \ $N = u_y$.
This corresponds to \[ \frac{\mathrm{d}y}{\mathrm{d}x} = \frac{-M}{N} = \frac{-u_x}{u_y}\]
with the arbitrary constant eliminated by the appropriate substitution. \par 
The \de\ of the orthogonal trajectory is \[ \frac{\mathrm{d}y}{\mathrm{d}x} = \frac{u_y}{u_x} = \frac{N}{M} \]
which can be written as $u_x\mathrm{d}y = u_y\mathrm{d}x$ or \[ - u_y\mathrm{d}x + u_x\mathrm{d}y = 0. \tag{4.2}\]
If $u(x,y,c_1)$ is harmonic, then (4.2) will be exact as we would have: \begin{align*}
(-u_y)_y = (u_x)_x & \Leftrightarrow -u_{yy} = u_{xx}\\
& \Leftrightarrow u_{xx} + u_{yy} = 0
\end{align*}
Furthermore, it is recognized that (4.2) is the exact conjugate of (4.1) as given in (3.10) and thus has its solution given by (3.7) as \[v(x,y,c_1) = \int u_x\partial y - \int \left[ u_y + \prt{}{x} \int u_x \partial y\right] \mathrm{d}x + c_1\]
which has been shown to be harmonic conjugate of $u(x,y,c)$. This means that the orthogonal trajectories of a harmonic function are its harmonic conjugates. \par\smallskip 
On the other hand, if $u(x,y,c)$ is not harmonic, then (4.2) will obviously not be exact, but can then be solved by other means to obtain the orthogonal trajectories.
The above results can be summarized in the theorem below:\par\smallskip 
{\bf Theorem 4.1}
\begin{enumerate}[(a)] 
\setlength{\itemsep}{0pt}
\item {\it Orthogonal trajectories of harmonic functions are necessarily harmonic.}
\item {\it Orthogonal trajectories of harmonic functions are their harmonic conjugates.}
\item \it {Any two functions that satisfy the Cauchy-Riemann equations and has non-vanishing first partial derivatives are orthogonal trajectories of each other.}
\end{enumerate}
\begin{proof}
Parts (a) and (b) have already been achieved. For (c), let $u = u(x,y),\ \ v = v(x,y)$ satisfy the Cauchy-Riemann equations; $u_x = v_y$ \ and \ $u_y = -v_x$, from the derivatives of implicit functions. we have that: \[ 
\bigg(\frac{\mathrm{d}y}{\mathrm{d}x}\bigg)_u = \frac{-u_x}{u_y} \quad \mbox{ and }\quad  \left(\frac{\mathrm{d}y}{\mathrm{d}x}\right)_v = \frac{-v_x}{v_y} \tag{4.3} \]
respectively, for $u$ and $v$. If $u,\ v$ are orthogonal trajectories of each other, then the product of the slopes in (4.3) above will be negative unity.\par\smallskip 
Assume for a contradiction that this is not the case, let \[
\left( \frac{\mathrm{d}y}{\mathrm{d}x}\right)_u\left(\frac{\mathrm{d}y}{\mathrm{d}x}\right)_v = k(x,y)\]
that is: \begin{align*}
\left(\frac{-u_x}{u_y}\right)\left(\frac{-v_x}{v_y}\right) &= \frac{u_xv_x}{u_yv_y} = k(x,y)\\
\Rightarrow u_xv_x &= k(x,y)u_yv_y\\ 
\mbox{ or }\quad \ u_xv_x &- k(x,y)u_yv_y = 0
\end{align*}
Let $v_x = -u_y$ and $v_y = u_x$ \qquad {\it (from the Cauchy-Riemann equations)}\\
then: \begin{align*}
u_x(-u_y) - k(x,y)u_y(u_x)& = 0\\
-u_xu_y - k(x,y)u_yu_x &=0\\
-u_xu_y (1 + k(x,y)) &= 0\\ 
\end{align*}
Since $u_x, u_y \neq 0$, it follows that \[ 1 + k(x,y) = 0 \mbox{ \quad or \quad } k(x,y) = -1.\]
Therefore,
\[\left( \frac{\mathrm{d}y}{\mathrm{d}x}\right)_u\left(\frac{\mathrm{d}y}{\mathrm{d}x}\right)_v = -1\tag{4.4}\] This proves that $u$ and $v$ are orthogonal trajectories of each other. \hfill \qed
\end{proof}
The above theorem in other words lends yet another way of looking at harmonic functions and their conjugates. It further provides another way of obtaining harmonic conjugates of harmonic functions, namely by seeking for their orthogonal trajectories. In the same way, we know that any two harmonic conjugates $u$ and $v$ satisfy the C-R equations, now we know that they also satisfy the equation (4.4). Therefore, given a harmonic function $u(x,y)$, to find its conjugate $v(x,y)$ from (4.4) now becomes a problem of solving the partial \de: 
\[u_xv_x + u_yv_y = 0\] 
where $u_x,\ u_y$ are known functions and $v$ is the unknown.
The solution is obtained by solving its characteristic equation: \[\frac{\mathrm{d}x}{u_x} = \frac{\mathrm{d}y}{u_y} = \frac{\mathrm{d}v}{0}\]
Looping the first two gives us the exact \de\ obtained in (3.10) whose solution is $v(x,y)$. Integrating the last term in the above equation yields $v = $ constant.\par 
Hence obtain the same result in (3.7) \[
v(x,y) = \int u_x\partial y - \int \left[ u_y + \prt{}{x}\int u_x \partial y\right] \mathrm{d}x =c\]
\vfill\newpage
\section*{References}
\begin{enumerate}[{[1]}]
\item S. Axter, P. Bourdon, W. Raney, {\it Harmonic Function Theory (Second Edition)} Springer-Verlag New York. Inc (2001).
\item W. E. Boyce, P. C. D. Prima, {\it Elementary Differential Equations and Boundary Value Problems (Seventh Edition)}, John Wiley \& Sons Inc. (2001).
\item E. D. Rainville, P. E. Bedient, {\it Elementary Differential Equations (Seventh Edition)}, Macmillan Publishing Company New York (1989).
\item D. G. Zill, {\it A First Course in Differential Equations with Applications (Second Edition),} PWS Publishers, Boston (1982).
\item D. G. Zill, P. D. Shanahan, {\it A First Course in Complex Analysis with Applications,} Jones and Bartlett Publishers Inc. (2003).
\end{enumerate}
\end{document}